\documentclass[12pt]{amsart}
\newtheorem{theorem}{Theorem}
\newtheorem*{theorem*}{Theorem}
\newtheorem{lemma}{Lemma}
\newtheorem*{lemma*}{Lemma}
\newtheorem{corollary}{Corollary}
\newtheorem{proposition}{Proposition}

\theoremstyle{definition}
\newtheorem{definition}{Definition}

\theoremstyle{remark}
\newtheorem{remark}{Remark}
\newtheorem*{remark*}{Remark}

\newcommand{\N}{\mathbb{N}}

\newcommand{\Z}{\mathbb{Z}}

\newcommand{\Zn}{\mathbb{Z}_n}

\usepackage{amssymb}
\usepackage{caption}
\usepackage{fullpage}
\usepackage{amsmath}
\usepackage{amsfonts}
\usepackage{latexsym}
\usepackage{graphicx}
\usepackage{tikz}
\usepackage{enumerate}
\usepackage[foot]{amsaddr}
\usepackage{hyperref}

\title{Finite Abelian Groups with Toroidal Subgroup Lattices}
\author{Richard A. Moy}
\address{Lee University, 1120 N. Ocoee St., Cleveland, TN 37311}

\begin{document}

\begin{abstract}
In this paper, we determine the genus of the subgroup lattice of several families of abelian groups. In doing so, we classify all finite abelian groups whose subgroup lattices can be embedded into the torus. 
\end{abstract}

\maketitle

\section{Introduction}

In every introductory abstract algebra course, students draw subgroup lattices, and they quickly learn that most subgroup lattices cannot be drawn without crossing edges. However, the groups whose subgroup lattice can be drawn on the plane without crossing edges were determined in \cite{ST04} and \cite{BR06}. Their results naturally lead us to the question of which groups have subgroup lattices that can be drawn on the $1$-torus without crossing edges. In \cite{BMS12}, the REU participants determined which cyclic groups have subgroup lattice that embed into the $1$-torus. In this paper, we classify the abelian groups whose subgroup lattice graphs are genus one. Following the notation in \cite{L20}, we formally define the subgroup lattice graph.

\begin{definition}
    The subgroup lattice graph $\Gamma(G)$ of a finite group $G$ is a graph whose vertices are the subgroups of the group and two vertices, $H_1$ and $H_2$ are connected by an edge if an only if $H_1\le H_2$ and there is no subgroup $K$ such that $H_1<K<H_2$. 
\end{definition}

\begin{definition}
The genus $\gamma(\Gamma)$ of a graph $\Gamma$ is the minimum $g$ such that there exists an embedding of $\Gamma$ into the orientable surface $S_g$ of genus $g$.
\end{definition}

Let $\Zn$ denote the cyclic group of size $n$. We now state the main result of this paper.

\begin{theorem}\label{theorem:main}
If $G$ is a finite abelian group, then $\gamma(\Gamma(G))=1$ if and only if
\begin{itemize}
    \item $G=\Zn$ where $n=p_1^{e_1}p_2^{e_2}p_3$ where $p_1,p_2,p_3$ are distinct primes and $(e_1,e_2)\in \{(2,2),(3,2),(3,3) \}$, or
    \item $G=\Zn$ where $n=p_1p_2p_3p_4$ where $p_1,p_2,p_3,p_4$ are distinct primes, or
    \item $G\in\{\Z_4\times \Z_4,\Z_8\times \Z_8,\Z_2\times \Z_2\times \Z_q,\Z_4\times \Z_2\times \Z_q \}$ where $q\ne 2$ is prime, or
    \item $G\in\{\Z_9\times \Z_9, \Z_3\times \Z_3\times \Z_q \}$ where $q\ne 3$ is prime, or
    \item $G=\Z_{25}\times \Z_{25}$.
\end{itemize}
\end{theorem}

\section{Cyclic Groups}
In \cite{ST04}, the authors determined all abelian groups whose subgroup lattice graphs are planar. For example, consider the subgroup lattice in Figure \ref{figure:cp13p22} of the cyclic group $\Z_{p_1^3p_2^2}$ where $p_1,p_2$ are distinct primes. The lattice is simply the lattice of divisors where two divisors are connected if and only if their ratio is prime. 

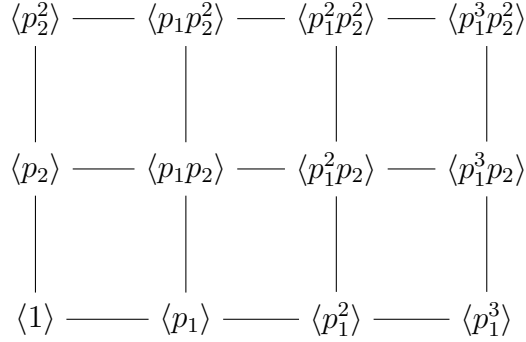
\begin{figure}[!h]
\begin{tikzpicture}[node distance=2cm]

\node(p13p22)                           {$\langle p_1^3p_2^2\rangle$};
\node(p12p22) [left of=p13p22] {$\langle p_1^2p_2^2\rangle$};
\node(p1p22) [left of=p12p22] {$\langle p_1p_2^2\rangle$};
\node(p22)      [left of=p1p22]       {$\langle p_2^2\rangle$};

\node(p13p2)      [below of=p13p22]      {$\langle p_1^3p_2\rangle$};
\node(p12p2)      [left of=p13p2]       {$\langle p_1^2p_2\rangle$};
\node(p1p2)      [left of=p12p2]       {$\langle p_1p_2\rangle$};
\node(p2)      [left of=p1p2]       {$\langle p_2\rangle$};

\node(p13)      [below of=p13p2]       {$\langle p_1^3\rangle$};
\node(p12)      [left of=p13]  {$\langle p_1^2\rangle$};
\node(p1)       [left of=p12] {$\langle p_1\rangle$};
\node(1)            [left of =p1]     {$\langle 1\rangle$};

\draw(1)       -- (p1);
\draw(p1)       -- (p12);
\draw(p12)       -- (p13);

\draw(p2)       -- (p1p2);
\draw(p1p2)       -- (p12p2);
\draw(p12p2)       -- (p13p2);

\draw(p22)       -- (p1p22);
\draw(p1p22)       -- (p12p22);
\draw(p12p22)       -- (p13p22);

\draw(1) -- (p2);
\draw(p2) -- (p22);

\draw(p1) -- (p1p2);
\draw(p1p2) -- (p1p22);

\draw(p12) -- (p12p2);
\draw(p12p2) -- (p12p22);

\draw(p13) -- (p13p2);
\draw(p13p2) -- (p13p22);

\end{tikzpicture}
\caption{Subgroup Lattice of $\Z_{p_1^3p_2^2}$}
\label{figure:cp13p22}
\end{figure}

We readily see that the subgroup lattice of $\Z_{p_1^3p_2^2}$ is simply the graph Cartesian product of the path graph on 3 vertices and the path graph on 4 vertices. This leads us to the following definition in which we follow the notation from \cite{MS22}.

\begin{definition}
A $k$-dimensional grid graph $\Gamma=\Gamma(e_1,\dots,e_k)$ is the graph Cartesian product of $k$ paths, where the $i^{th}$ path has length $e_i$. 
\end{definition}

From this definition, the following lemma follows immediately.

\begin{lemma}
The subgroup lattice of the cyclic group $\Z_{p_1^{e_1}\dots p_k^{e_k}}$ where the $p_i$ are distinct primes and $e_i\in\N$ is precisely $\Gamma(e_1,\dots,e_k)$.
\end{lemma}

In \cite{MS22}, the authors prove many results about the genus of grid graphs $\Gamma(e_1,\dots,e_k)$. In particular, they prove the following result which lists all grid graphs with genus $0$ or $1$. The first two bullet points of the below theorem had already been discovered by Starr and Turner \cite{ST04} and the last two bullet points had been discovered by Berry, McGinnis, and Sanchez \cite{BMS12}. However, these earlier results used the language of subgroup lattice graphs of cyclic groups as opposed to grid graphs.

\begin{proposition}[\cite{BMS12,MS22,ST04}]\label{prop:cyclic_genus_one}
A grid graph $\Gamma(e_1,\dots,e_k)$ embeds on the torus if and only if
\begin{itemize}
    \item $k\le 2$,
    \item $k=3$ and at most one of $e_1,e_2,e_3$ is greater than 1,
    \item $k=3$ and $(e_1,e_2,e_3)\in\{(2,2,1),(3,2,1),(3,3,1)\}$, or
    \item $k=4$ and $(e_1,e_2,e_3,e_4)=(1,1,1,1)$.
\end{itemize}
\end{proposition}

Thus, all cyclic groups whose subgroup lattice graphs are genus zero (the first two bullet points above) or one (the last two bullet points above) have been classified.

\subsection{Tools for Finding the Genus of Subgroup Lattice Graphs of Cyclic Groups}
Although Proposition \ref{prop:cyclic_genus_one} resolves the question of which cyclic groups have subgroup lattices of genus one, there are several results in \cite{MS22} that let us determine the genus of the subgroup lattice of cyclic groups in other cases. One of the most useful results is the exact genus of grid graphs $\Gamma(e_1,\dots,e_k)$ for which at least 3 of the $e_i$ are odd. Let $E(\Gamma)$ and $V(\Gamma)$ represent the number of edges and vertices respectively in a grid graph $\Gamma$. The following result is due to White \cite{W70}; however, it has been rephrased in the language of \cite{MS22}.

\begin{proposition}[Theorem 4,\cite{W70}]
    Let $\Gamma=\Gamma(e_1,\dots,e_k)$ be a $k$-dimensional grid graph where $k\ge 3$ and at least three $e_i$ are odd, for $i=1,\dots,k$. Then $G$ is a lower-embeddable grid graph, i.e. 
    $$
    \gamma(\Gamma)=1+\frac{|E(\Gamma)|}{4}-\frac{|V(\Gamma)|}{2}=1+\frac{1}{2}\prod_{i=1}^k{(e_i+1)}\left[ \frac{1}{2}\sum_{i=1}^k{\frac{e_i}{e_i+1}-1}\right].
    $$
\end{proposition}

This leads to the immediate corollary.

\begin{corollary}
    If $G=\Z_{p_1^{e_1}\dots p_k^{e_k}}$ is a cyclic group with $k\ge 3$ and $e_i$ odd for at least three $i$, then:
    \begin{equation}\label{eq:genus_formula_odds}
    \gamma(\Gamma(G))=1+\frac{1}{2}\prod_{i=1}^k{(e_i+1)}\left[ \frac{1}{2}\sum_{i=1}^k{\frac{e_i}{e_i+1}-1}\right].
    \end{equation}
\end{corollary}
Here is an easy consequence of the previous result.
\begin{corollary}
    Consider the subgroup lattice graph of the cyclic group $G=\Z_{p_1^np_2p_3p_4}$ where $p_i$ are distinct primes. Then 
    \begin{equation}\label{eq:genus_formula_n111}
    \gamma(\Gamma(G))=1+\frac{1}{2}\cdot (n+1)\cdot 2^3\cdot \left[\frac{1}{2}\left(\frac{n}{n+1}+3\cdot \frac{1}{2}\right)-1\right]=n.
    \end{equation}
\end{corollary}

Another easy consequence is the genus of the $k$-hypercube, which is the subgroup lattice graph of $\Z_{p_1\dots p_k}$ where $p_i$ are distinct primes.
\begin{corollary}
    Consider the subgroup lattice graph of the cyclic group $G=\Z_{p_1\dots p_k}$ where $p_i$ are distinct primes and $k\ge 3$. Then
    \begin{equation}\label{eq:genus_formula_hypercube}
    \gamma(\Gamma(G))=1+\frac{1}{2}\cdot 2^k\cdot \left[\frac{1}{2}\sum_{i=1}^k{\frac{1}{2}}-1 \right]=1+2^{k-3}(k-4).
    \end{equation}
\end{corollary}

We now state two additional results from \cite{MS22} that will be needed in the proof of Theorem \ref{theorem:main}.

\begin{proposition}[Theorem 3, \cite{MS22}]
For any $e_1,e_2\in\N$, we have 
\begin{equation}\label{eq:genus_formula_1}
\gamma(\Gamma(e_1,e_2,1))=\left\lfloor \frac{e_1}{2}\right\rfloor \left\lfloor \frac{e_2}{2}\right\rfloor.
\end{equation}
\end{proposition}

\begin{proposition}[Theorem 4, \cite{MS22}]
For any $e_1\in \N$, we have that 
\begin{equation}\label{eq:genus_formula_22}
\gamma(\Gamma(e_1,2,2))=e_1.
\end{equation}
\end{proposition}

\subsection{Lower and Upper Bounds on the Genus of a Grid Graph}
While the previous section provided exact formulas for the genus of various grid graphs, there are several inequalities for the genus that will be useful. The first result is a lower bound for the genus of any grid graph.

\begin{proposition}[Corollary 1, \cite{MS22}]
    Let $\Gamma(e_1,\dots, e_k)$ be a grid graph with $k>1$. Then
    \begin{equation}\label{eq:lower_bound}
    \gamma(\Gamma(e_1,\dots,e_k))\ge 1+\frac{|E|}{4}-\frac{|V|}{2}=1+\frac{1}{2}\prod_{i=1}^k{(e_i+1)}\left[ \frac{1}{2}\sum_{i=1}^k{\frac{e_i}{e_i+1}-1}\right]
    \end{equation}
    where $|E|$ and $|V|$ are the number of edges and vertices respectively in $\Gamma(e_1,\dots,e_k)$.
\end{proposition}

\begin{proposition}[Proposition 3, \cite{MS22}]\label{prop:upper_bounds}
Let $\Gamma=\Gamma(e_1,e_2,e_3)$ be a $3$-dimensional grid with at least one even grid parameter.
\begin{enumerate}[(a)]
    \item If exactly one grid parameter is even (the third parameter), then 
    \begin{equation}\label{eq:upper_bound}
        \gamma(\Gamma(e_1,e_2,e_3)\le \gamma(\Gamma(e_1,e_2,e_3-1))+\frac{(e_1+1)(e_2+1)}{4}-1.
    \end{equation}

    \item If exactly two grid parameters are even (the second \& third parameters), then 
        \begin{equation*}
            \gamma(\Gamma(e_1,e_2,e_3)\le \gamma(\Gamma(e_1,e_2-1,e_3-1))+\frac{(e_1+1)(e_2+e_3)}{4}-1.
        \end{equation*}

    \item If all three grid parameters are even, then
        \begin{equation*}
            \gamma(\Gamma(e_1,e_2,e_3))\le \gamma(\Gamma(e_1-1,e_2-1,e_3-1))+\frac{e_1e_2+e_1e_3+e_2e_3}{4}-1
        \end{equation*}
\end{enumerate}

\end{proposition}

While these formulas are useful, one sometimes needs to show that a graph has a lower bound for its genus by finding a graph minor with a particular genus. When one combines multiple copies of the same graph, e.g. $K_{3,3}$ one can use the block decomposition of a graph to compute its genus. We follow the discussion of block decompositions found in \cite{MS22}.

\begin{definition}
    A {\bf block} of a graph $\Gamma$ is a maximal $2$-connected subgraph $B$ of $\Gamma$. Given a connected graph $\Gamma$, there exists a unique collection of blocks $\mathcal{B}=\{B_1,\dots,B_k\}$ such that $\cup_{i=1}^k{B_i}=\Gamma$, which is called the {\bf block decomposition} of $\Gamma$.
\end{definition}

The following theorem connecting the genus of a graph to the genus of the blocks in its block decomposition is due to Battle, Harary, Kodama, and Youngs \cite{BHKY62}.

\begin{theorem}[Theorem 1, \cite{BHKY62}]\label{theorem:block}
If $\Gamma$ is a graph with block decomposition $\{B_1,\dots,B_k\}$, then $\gamma(\Gamma)=\sum_{i=1}^k{\gamma(B_i)}$.
\end{theorem}

Using these results, we can prove the following theorem.

\begin{theorem}\label{theorem:cyclic_genus}
Let $G=\Zn$ be a cyclic group with subgroup lattice graph $\Gamma(G)$ where $n=p_1^{e_1}\dots p_k^{e_1}$ with $p_i$ distinct primes. Then $\gamma(\Gamma(G))=2$ if and only if
\begin{itemize}
    \item $k=3$ and $(e_1,e_2,e_3)\in\{(4,2,1),(4,3,1),(5,2,1),(5,3,1),(2,2,2)\}$ or
    \item $k=4$ and $(e_1,e_2,e_3,e_4)=(2,1,1,1)$.
\end{itemize}
We have $\gamma(\Gamma(G))=3$ if and only if
\begin{itemize}
    \item $k=3$ and $(e_1,e_2,e_3)\in\{(6,2,1),(6,3,1),(7,2,1),(7,3,1),(3,2,2)\}$ or
    \item $k=4$ and $(e_1,e_2,e_3,e_4)=(3,1,1,1)$.
\end{itemize}
We have $\gamma(\Gamma(G))=4$ if
\begin{itemize}
    \item $k=3$ and $$(e_1,e_2,e_3)\in\{((8,2,1),(8,3,1),(9,2,1),(9,3,1),(4,4,1),(5,4,1),(5,5,1),(4,2,2),(3,3,2)\}$$ or
    \item $k=4$ and $(e_1,e_2,e_3,e_4)=(4,1,1,1)$.
\end{itemize}
\end{theorem}

\begin{remark}
The above list of cyclic groups with genus $4$ subgroup lattice graphs is exhaustive with the possible exception of $\Z_{p_1^2p_2^2p_3p_4}$. Using \eqref{eq:lower_bound}, one can show it has genus at least $4$, and using other results from \cite{MS22}, one can show it has genus at most $6$. The exact genus is unknown to the author.
\end{remark}

\begin{proof}
Using equations \eqref{eq:genus_formula_1}, \eqref{eq:genus_formula_22}, and \eqref{eq:genus_formula_n111}, we can easily show that all the subgroup lattice graphs above have the specified genus with the exception of $G=\Z_{p_1^3p_2^3p_3^2}$ where $p_i$ are distinct primes. 

Using \eqref{eq:upper_bound}, we see that $\gamma(\Gamma(\Z_{p_1^3p_2^3p_3^2}))\le 4$. Now, we must show it is at least four. Below, we have drawn a graph that is combination of four $K_{3,3}$ graphs.

\begin{center}
\includegraphics[width=0.5\linewidth]{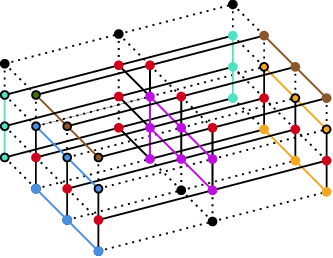}
\end{center}

Omit all the dotted edges and black vertices. Contract the blue, green, purple, brown, and orange paths to single vertices. Note that there are two separate blue paths, one with solid blue vertices and one with blue vertices outlined with a black circle. These paths are contracted to two different vertices. The same phenomenon occurs with the green, orange, and brown vertices. Doing so results in the following minor of $\Gamma(\Z_{p_1^3p_2^3p_3^2})$.

\begin{center}
\includegraphics[width=0.25\linewidth]{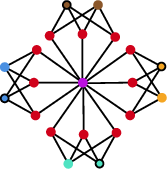}
\end{center}

Using Theorem \ref{theorem:block} and the fact that $\gamma(K_{3,3})=1$ we deduce that the genus of the above graph minor is $4$. Hence, we have deduced that $\gamma(\Gamma(\Z_{p_1^3p_2^3p_3^2}))=4.$

Now, we must show that cyclic groups not appearing in the list above or in the list from Proposition \ref{prop:cyclic_genus_one} have genus at least 5 with the exception of $\Gamma(\Z_{p_1^2p_2^2p_3p_4})$. For the remainder of the proof, suppose that $G$ is a cyclic group \underline{not} appearing in the list of groups in Proposition \ref{prop:cyclic_genus_one} or Theorem \ref{theorem:cyclic_genus}.

Suppose $k=3$ and $e_1\ge e_2\ge e_3$. 

If $e_3=1$ and $e_2=2$ or $3$, then $e_1\ge 10$ and thus $\gamma(\Gamma(G))\ge 5$ by \eqref{eq:genus_formula_1}. If $e_3=1$ and $e_2=4$ or $5$, then $e_1\ge 6$ and thus $\gamma(\Gamma(G))\ge 6$ by \eqref{eq:genus_formula_1}. If $e_3=1$ and $e_2\ge 6$ then $e_1\ge 6$ and thus $\gamma(\Gamma(G))\ge 9$ by \eqref{eq:genus_formula_1}.

If $e_3=2$ and $e_2=2$ then $e_1\ge 5$ and thus $\gamma(\Gamma(G))\ge 5$ by \eqref{eq:genus_formula_22}. If $e_3=2$ and $e_2\ge 3$ then $e_1\ge 4$ and thus $\gamma(\Gamma(G))>4$ by \eqref{eq:lower_bound}. 

If $e_3\ge 3$ then $e_1,e_2\ge 3$. Since $\gamma(\Gamma(\Z_{p_1^3p_2^3p_3^3}))=5$ using \eqref{eq:genus_formula_odds}, we deduce that $\gamma(\Gamma(G))\ge 5$.

Suppose $k=4$ and $e_1\ge e_2\ge e_3\ge e_4$. If $e_2=e_3=e_4=1$, then $e_1\ge 5$ and thus $\gamma(\Gamma(G))\ge 5$ by \eqref{eq:genus_formula_n111}. If $e_3=e_4=1$ and $e_2=2$ then $e_1\ge 3$ and thus $\gamma(\Gamma(G))\ge 6$ by \eqref{eq:lower_bound}. If $e_3=e_4=1$ and $e_2\ge 3$ then $e_1\ge 3$ and $\gamma(\Gamma(G))\ge 9$ by \eqref{eq:lower_bound}. If $e_3\ge 2$ then $e_1\ge e_2\ge 2$ and thus $\gamma(\Gamma(G))>7$ by \eqref{eq:lower_bound}.

Suppose $k=5$ and $e_1\ge e_2\ge e_3\ge e_4\ge e_5$. Then we know $\gamma(\Gamma(G))\ge 5$ since $\gamma(\Gamma(\Z_{p_1p_2p_3p_4p_5}))=5$ by \eqref{eq:genus_formula_hypercube}.
\end{proof}

Thus, with the exception of $\Z_{p_1^2p_2^2p_3p_4}$, we have completed the classification of cyclic groups with subgroup lattice graphs of genus up to $4$.

\section{Non-Cyclic Abelian Groups}

Now that we have classified all cyclic groups whose subgroup lattice graph has genus up to 4 with the exception of $\Z_{p_1^2p_2^2p_3p_4}$, we proceed with determining the non-cyclic abelian groups whose subgroup lattice has genus one. The following theorem classifies all non-cyclic abelian groups whose subgroup lattices have genus zero. The planar cyclic groups are described in the first two bullet points of Proposition \ref{prop:cyclic_genus_one}.

\begin{theorem}[Theorem 14, \cite{ST04}]\label{theorem:genus_zero}
Let $G$ be a finite abelian group. Then $\Gamma(G)$ is planar if and only if $G$ is isomorphic to a planar cyclic group or to $\Z_{p^a}\times \Z_p$ where $p$ is a prime and $a\in\N$.
\end{theorem}

First, we examine a few families of graphs that will be needed to prove Theorem \ref{theorem:main}.

\begin{lemma}
    Let $G_n$ be the graph pictured below. 

    \begin{center}
    \includegraphics[width=0.25\linewidth]{"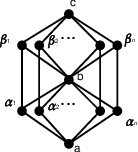"}
\end{center}
The genus of $G_n$ is $\gamma(G_n)=\lceil \frac{n-2}{4}\rceil$.
\end{lemma}

\begin{proof}
    The graph $G_n$ has a clear $K_{3,n}$ minor and $\gamma(K_{3,n})=\lceil \frac{n-2}{4}\rceil$. Hence, $\gamma(G_n)\ge \lceil\frac{n-2}{4}\rceil$. To demonstrate that $G_n$ can be drawn on a surface of genus $\lceil \frac{n-2}{4}\rceil$, we will construct a set of cycles that traverse each edge in both directions (e.g. define the faces of the embedding of the graph on a surface). We will first consider the case when $n\equiv 2 \mod 4$. The cycles are as follows:
    \begin{itemize}
        \item $(b,\alpha_1,\beta_1)$, $(b,\beta_2,\alpha_2)$, $(b,\alpha_3,\beta_3)$, $(b, \beta_4,\alpha_4)$, $\dots$, $(b,\beta_n,\alpha_n)$ 
        \item $(b,\beta_1,c,\beta_2)$, $(b,\beta_3,c,\beta_4)$, $\dots$, $(b,\beta_{n-1},c,\beta_n)$
        \item $(b,\alpha_2,a,\alpha_3)$, $(b,\alpha_4,a,\alpha_5)$, $\dots$, $(b,\alpha_n,a,\alpha_1)$ 
        \item $(c,\beta_1,\alpha_1,a,\alpha_{1+\frac{n}{2}},\beta_{1+\frac{n}{2}})$, $(c,\beta_3,\alpha_3,a,\alpha_{3+\frac{n}{2}},\beta_{3+\frac{n}{2}})$, $\dots$, $(c,\beta_{\frac{n}{2}},\alpha_{\frac{n}{2}},a,\alpha_n,\beta_n)$,\newline  $(c,\beta_{\frac{n}{2}+2},\alpha_{\frac{n}{2}+2},a,\alpha_2,\beta_2)$, $\dots$, $(c,\beta_{n-1},\alpha_{n-1},a,\alpha_{\frac{n}{2}-1},\beta_{\frac{n}{2}-1})$
    \end{itemize}

Using these cycles, observe that this embedding of the graph $G_n$ into some surface has $\frac{5n}{2}$ faces, $2n+3$ vertices, and $5n$ edges. Using Euler's formula, we deduce that $2-2g=\frac{5n}{2}-5n+2n+3=-\frac{n}{2}+3.$ This equation yields $g=\frac{n-2}{4}=\lceil \frac{n-2}{4}\rceil$ since $n\equiv 2\mod 4$. Observe that the fact that $n\equiv 2\mod 4$ is used inherently in the definition of the $6$-cycles in the fourth bullet point.

Thus we have proved our result in the case when $n\equiv 2\mod 4$. The other cases follow from the fact that $G_n$ is a subgraph of $G_{n+1}$, $G_{n+2}$ and $G_{n+2}$. If $n\not\equiv 2\mod 4$, there exists a unique integer $m\in\{n+1,n+2,n+3\}$ satisfying $m\equiv 2\mod 4$ with $\lceil\frac{n-2}{4}\rceil=\lceil\frac{m-2}{4}\rceil$.
\end{proof}

We then use this result to prove the genus of $\Gamma(\Z_{p^2}\times\Z_{p^2})$.

\begin{proposition}\label{prop:Zp2_Zp2}
    For a prime $p$, $\gamma(\Gamma(\Z_{p^2}\times \Z_{p^2}))=\lceil \frac{p-1}{4}\rceil$.
\end{proposition}
\begin{proof}
The subgraph lattice graph of $\Z_{p^2}\times \Z_{p^2}$ is pictured below in Figure \ref{fig:Zp2_Zp2}.
\begin{figure}[!h]
    \includegraphics[width=0.25\linewidth]{"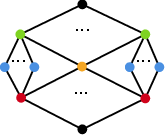"}
    \caption{Subgroup Lattice Graph of $\Z_{p^2}\times \Z_{p^2}$}
    \label{fig:Zp2_Zp2}
\end{figure}
The red vertices represent the $p+1$ subgroups of size $p$. The blue vertices represent the $p^2+p$ cyclic subgroups of size $p^2$ and the orange vertex represents the unique non-cyclic subgroup of size $p^2$. The green vertices represent the $p+1$ subgroups of size $p^3$.

It is easy to see that $\Gamma(\Z_{p^2}\times \Z_{p^2})$ has a $G_{p+1}$ minor by merging all the blue vertices with the respective red vertex they connect to. Hence, $\gamma(\Gamma(\Z_{p^2}\times \Z_{p^2}))\ge \lceil \frac{p-1}{4}\rceil$. It is also relatively easy to see that since $G_{p+1}$ embeds onto a surface of genus $\lceil \frac{p-1}{4}\rceil$, we can also embed $\Gamma(\Z_{p^2}\times \Z_{p^2})$ on the same surface. This is because one can cut out a small neighborhood of the edges $\alpha_1\beta_1,\dots \alpha_{p+1}\beta_{p+1}$ and replace it by the corresponding red vertex, green vertex, and $p$ blue vertices along with the edges that connect them. For example, see Figure \ref{fig:neighborhood1}.

\begin{figure}[!h]
    \includegraphics[width=0.25\linewidth]{"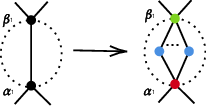"}
    \caption{}
    \label{fig:neighborhood1}
\end{figure}

Hence, we have proven that $\gamma(\Gamma(\Z_{p^2}\times \Z_{p^2}))=\lceil \frac{p-1}{4}\rceil$.
\end{proof}

\begin{lemma}
    Let $H_n$ be the graph pictured below. 

    \begin{center}
    \includegraphics[width=0.5\linewidth]{"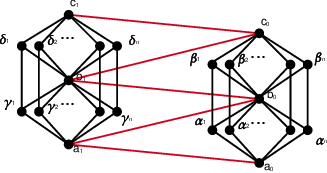"}
\end{center}
If $n$ is odd, the genus of $H_n$ is $\gamma(H_n)=2\lceil \frac{n-2}{4}\rceil$.
\end{lemma}

\begin{proof}
    Observe that $H_n$ contains two copies of $G_n$. By deleting four of the red edges and using Theorem \ref{theorem:block}, one deduces that $\gamma(H_n)\ge 2\lceil \frac{n-2}{4}\rceil$.

    Now, we consider the case when $n\equiv 1\mod 4$. To demonstrate that $H_n$ can be drawn on a surface of genus $2\lceil \frac{n-2}{4}\rceil$, we will construct a set of cycles that traverse each edge in both directions exactly once. The cycles are as follows:
    \begin{enumerate}[(a)]
    \item $(b_0,\alpha_1,\beta_1)$, $(b_0,\beta_2,\alpha_2)$, \dots, $(b_0,\beta_{n-1},\alpha_{n-1})$, $(b_0,\alpha_n,\beta_n)$
    \item $(b_1,\delta_1,\gamma_1)$, $(b_1,\gamma_2,\delta_2)$, \dots, $(b_1,\gamma_{n-1},\delta_{n-1})$, $(b_1,\delta_n,\gamma_n)$
    \item $(b_0,\beta_1,c_0,\beta_2)$, $(b_0,\alpha_2,a_0,\alpha_3)$, \dots, $(b_0,\beta_{n-2},c_0,\beta_{n-1})$, $(b_0,\alpha_{n-1},a_0,\alpha_n)$
    \item $(b_1,\delta_2,c_1,\delta_1)$, $(b_1,\gamma_3,a_1,\gamma_2)$, \dots, $(b_1,\delta_{n-2},c_1,\delta_{n-1})$, $(b_1,\gamma_n,a_1,\gamma_{n-1})$
    \item $(c_0,\beta_1,\alpha_1,a_0,\alpha_{(n+3)/2},\beta_{(n+3)/2})$, $(a_0,\alpha_2,\beta_2,c_0,\beta_{(n+5)/2},\alpha_{(n+5)/2})$, \dots, \newline$(c_0,\beta_{(n-3)/2},\alpha_{(n-3)/2},a_0,\alpha_{n-1},\beta_{n-1})$, $(a_0,\alpha_{(n-1)/2},\beta_{(n-1)/2},c_0,\beta_n,\alpha_n)$
    \item $(a_1,\gamma_1,\delta_1,c_1,\delta_{(n+3)/2},\gamma_{(n+3)/2})$, $(c_1,\delta_2,\gamma_2,a_1,\gamma_{(n+5)/2},\delta_{(n+5)/2})$, \dots, \newline $(a_1, \gamma_{(n-3)/2},\delta_{(n-3)/2},c_1,\delta_{n-1},\gamma_{n-1})$, $(c_1,\delta_{(n-1)/2},\gamma_{(n-1)/2},a_1,\gamma_n,\delta_n)$
    \item $(b_0,a_a,a_0,\alpha_1)$, $(b_0,\beta_n,c_0,b_1)$, $(b_0,b_1,\gamma_1,a_1)$, $(c_1,\delta_n, b_1,c_0)$
    \item $(c_0, \beta_{(n+1)/2},\alpha_{(n+1)/2},a_0,a_1,\gamma_{(n+1)/2},\delta_{(n+1)/2},c_1)$    
    \end{enumerate}

    Using these cycles, observe that this embedding of the graph $H_n$ into some surface has $5n+2$ faces, $4n+6$ vertices, and $10n+5$ edges. Using Euler's formula, we deduce that $2-2g=5n+2-(10n+5)+4n+6=-n+3.$ This equation yields $g=\frac{n-1}{2}=2\lceil \frac{n-2}{4}\rceil$ since $n\equiv 1\mod 4$. Observe that we inherently used the fact that $n\equiv 1\mod 4$ when we constructed the $6$-cycles in bullets (e) and (f) since we need $\frac{n-1}{2}$ to be even to have an equal number of cycles beginning with $c_0$ versus $a_0$ (or $a_1$ versus $c_1$) in our list of $6$-cycles.
    
    Thus we have proved our result in the case when $n\equiv 1\mod 4$. The case when $n\equiv 3\mod 4$ follows from the fact that $H_n$ is a subgraph of $H_{n+2}$.
\end{proof}

\begin{proposition}\label{prop:Zp3_Zp2}
    For an odd prime $p$, $\gamma(\Gamma(\Z_{p^3}\times \Z_{p^2}))=2\lceil \frac{p-2}{4}\rceil$. When $p=2$, $\gamma(\Gamma(\Z_{8}\times \Z_{4}))=1$.
\end{proposition}
\begin{proof}
The subgroup lattice of $\Z_{p^3}\times \Z_{p^2}$ is pictured below in Figure \ref{fig:Zp3_Zp2}.

\begin{figure}[!h]
    \includegraphics[width=0.5\linewidth]{"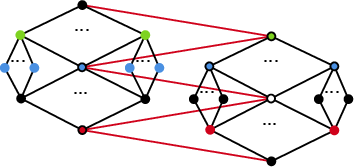"}
    \caption{Subgroup Lattice Graph of $\Z_{p^3}\times \Z_{p^2}$}
    \label{fig:Zp3_Zp2}
\end{figure}

The red vertices represent the $p+1$ subgroups of size $p$, and the red vertex with a black outline is the subgroup generated by $(p^2,0)$. The black vertices in the third row from the bottom represent the $p(p+1)$ cyclic subgroups of size $p^2$, and the white vertex represents the unique non-cyclic subgroup of size $p^2$. The blue vertices without a black outline represent the $p^2$ cyclic subgroups of size $p^3$ and the $p+1$ blue vertices with a black outline represent the $p+1$ non-cyclic subgroups of size $p^3$. Lastly, the green vertices are the $p+1$ subgroups of size $p^4$, and the green vertex outlined in black unique subgroup of size $p^4$ that is entirely $p^2$ torsion.

It is relatively easy to see that $\Gamma(\Z_{p^3}\times \Z_{p^2})$ has a $H_p$ minor. Hence, $\gamma(\Gamma(\Z_{p^3}\times \Z_{p^2}))\ge 2\lceil\frac{p-2}{4}\rceil$.  It is also relatively easy to see that since $H_n$ embeds onto a surface of genus $2\lceil \frac{p-2}{4}\rceil$, we can embed $\Gamma(\Z_{p^3}\times \Z_{p^2})$ on the same surface. This is because one can cut out a small neighborhood of the edges $\gamma_1\delta_1,\dots \gamma_{p+1}\delta_{p+1}$ and replace it by the corresponding black vertex, green vertex, and $p$ blue vertices along with the edges that connect them. For example, see Figure \ref{fig:neighborhood2} below.

\begin{figure}[!h]
    \includegraphics[width=0.25\linewidth]{"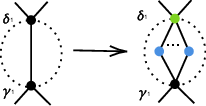"}
    \caption{}
    \label{fig:neighborhood2}
\end{figure}

A similar operation allows one to cut out a small neighborhood of the edges $\alpha_1\beta_1,\dots \alpha_{p+1}\beta_{p+1}$ and replace it by the corresponding red vertex, blue vertex, and $p$ black vertices along with the edges that connect them.

Now consider the case when $p=2$. We know that $\gamma(\Gamma(\Z_{8}\times \Z_{4}))\ge 1$ by Theorem \ref{theorem:genus_zero}. To demonstrate that $\gamma(\Gamma(\Z_{8}\times \Z_{4}))=1$, we simply need to show that we can draw $\Gamma(\Z_8\times \Z_4)$ on the 1-torus which we do here in Figure \ref{fig:Z8_Z4}. 

\begin{figure}[!h]
    \includegraphics[width=0.35\linewidth]{"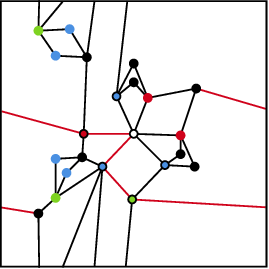"}
    \caption{Subgroup Lattice of $\Z_8\times \Z_4$ Drawn on the Torus}
    \label{fig:Z8_Z4}
\end{figure}
\end{proof}

\begin{proposition}\label{prop:Zp_Zp_Zq}
 The genus of $\Gamma(\Z_p\times \Z_p\times \Z_q)$ for primes $p\ne q$ is $\lceil \frac{p-1}{2}\rceil$.
\end{proposition}

\begin{proof}
The subgroup lattice of $\Z_p\times \Z_p\times \Z_q$ is pictured below in Figure \ref{fig:Zp_Zp_Zq}. 

\begin{figure}[!h]
    \includegraphics[width=0.25\linewidth]{"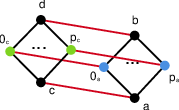"}
    \caption{Subgroup Lattice Graph of $\Z_p\times \Z_p\times \Z_q$}
    \label{fig:Zp_Zp_Zq}
\end{figure}

The blue vertices represent the $p+1$ subgroups of size $p$, the $a$ vertex represents the trivial subgroup, the $b$ vertex represents the unique subgroup of size $p^2$, the $c$ vertex represents the unique subgroup of size $q$, the green vertices represent the $p+1$ subgroups of size $pq$ and the vertex $d$ is the full group.

If one contracts the edges $0_a0_c,\dots, p_ap_c$ and deletes edges $bd$ and $ac$, then one obtains a $K_{4,p+1}$ graph with genus $\lceil \frac{(4-2)(p+1-2)}{4}\rceil=\lceil \frac{p-1}{2}\rceil$. Hence, $\gamma(\Gamma(\Z_p\times \Z_p\times \Z_q))\ge \lceil \frac{p-1}{2}\rceil$.

First, we will prove the case when $p$ is an $odd$ prime.  To show that we can embed the subgroup lattice on a surface of genus $\lceil \frac{p-1}{2}\rceil$, we will exhibit a set of cycles that traverse each edge in both directions exactly once. The cycles are as follows:
    \begin{itemize}
        \item $(a,0_a,b,1_a)$, $(a,2_a,b,3_a)$, $\dots$, $(a,(p-1)_a,b,p_a)$
        \item $(c,1_c,d,0_c)$, $(c,3_c,d,2_c)$, $\dots$, $(c,p_c,d, (p-1)_c)$
        \item $(a,c,0_c,0_a)$, $(a,p_a,p_c,c)$
        \item $(b,d,1_c,1_a)$, $(b,2_a,2_c,d)$
        \item $(a,1_a,1_c,c,2_c,2_a)$, $(a,3_a,3_c,c,4_c,4_a)$, $\dots$, $(a,(p-2)_a,(p-2)_c,c,(p-1)_c,(p-1)_a)$
        \item $(b,0_a,0_c,d,p_c,p_a)$, $(b,4_a,4_c,d,3_c,3_a)$, $\dots$, $(b,(p-1)_a,(p-1)_c,d,(p-2)_c,(p-2))a)$
    \end{itemize}

Using these cycles, observe that this embedding of $\Gamma(\Z_p\times \Z_p\times \Z_q)$ into some surface has $2p+4$ faces, $2(p+3)$ vertices, and $5p+7$ edges. Using Euler's formula, we deduce that $2-2g=2p+4-(5p+7)+2p+6=-p+3$. This equation yields $g=\frac{p-1}{2}=\lceil \frac{p-1}{2}\rceil$ since $p$ is an odd prime. Hence, we have shown that $\gamma(\Gamma(\Z_p\times \Z_p\times \Z_q))=\lceil \frac{p-1}{2}\rceil$ when $p$ is an odd prime. 

When $p=2$, we know that $\gamma(\Gamma(\Z_2\times \Z_2\times \Z_q))>0$ by Theorem \ref{theorem:genus_zero}. Since, $\Gamma(\Z_2\times \Z_2\times \Z_q)$ is a subgraph of $\Gamma(\Z_3\times \Z_3\times \Z_q)$, we deduce that  $\gamma(\Gamma(\Z_2\times \Z_2\times \Z_q))=1=\lceil \frac{2-1}{2}\rceil$.
\end{proof}

Now, we compile our results to prove that all the graphs in the statement of Theorem \ref{theorem:main} have genus one.

\begin{proposition}\label{prop:non_cyclic_genus_one}
The following groups have subgroup lattices with genus one.
\begin{itemize}
    \item $G\in\{\Z_4\times \Z_4,\Z_8\times \Z_4,\Z_2\times \Z_2\times \Z_p,\Z_4\times \Z_2\times \Z_q \}$ where $q\ne 2$ is prime, or
    \item $G\in\{\Z_9\times \Z_9, \Z_3\times \Z_3\times \Z_q \}$ where $q\ne 3$ is prime, or
    \item $G=\Z_{25}\times \Z_{25}$.
\end{itemize}
\end{proposition}

\begin{proof}
All the groups listed in this proposition, with the exception of $\Z_4\times \Z_2\times \Z_q$, have already been shown to have subgroup lattice graphs of genus one via Propositions \ref{prop:Zp2_Zp2}, \ref{prop:Zp3_Zp2}, and \ref{prop:Zp_Zp_Zq}. Thus all that remains is to show the same result for $\Z_4\times \Z_2\times \Z_q$. Since it does not appear in the list of planar groups in Theorem \ref{theorem:genus_zero}, it must have genus at least one. So all we must do is show that its subgroup lattice graph can be drawn on the $1$-torus

The subgroup lattice of $\Z_4\times \Z_2\times \Z_q$ is below in Figure \ref{fig:Z4_Z2_Zq}. The blue vertices represent the three cyclic subgroups of size $2$ and the blue vertex outlined in black is the subgroup generated by $(2,0,0)$. The red vertex is the unique non-cyclic subgroup of size $4$. The green vertices represent the three cyclic subgroups of size $2q$ and the green vertex outlined in black is the subgroup generated by $(2,0,1)$. The orange vertex is the unique non-cyclic subgroup of size $4q$.

\begin{figure}[!h]
    \includegraphics[width=0.25\linewidth]{"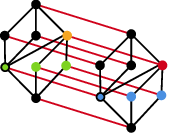"}
    \caption{Subgroup Lattice Graph of $\Z_4\times \Z_2\times \Z_q$}
    \label{fig:Z4_Z2_Zq}
\end{figure}

Figure \ref{fig:Z4_Z2_Zq_torus} shows how $\Gamma(\Z_4\times \Z_2\times \Z_q)$ can be drawn on the $1$-torus.

\begin{figure}[!h]
    \includegraphics[width=0.35\linewidth]{"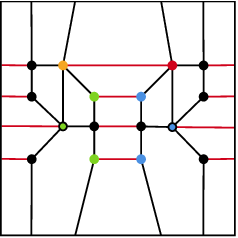"}
    \caption{Subgroup Lattice Graph of $\Z_4\times \Z_2\times \Z_q$ Drawn on the Torus}
    \label{fig:Z4_Z2_Zq_torus}
\end{figure}

\end{proof}

\subsection{Eliminating Graphs}
In this subsection, we will determine why various families of graphs have genus at least two and thus cannot be embedded into the 1-torus. First, recall that the {\bf girth} of a graph is the length of its shortest cycle. Now, we recall the following result. 

\begin{proposition}[Proposition 1, \cite{MS22}]
    Suppose $\Gamma$ is a connected graph whose girth is at least $4$. Then $\gamma(\Gamma)\ge 1+\frac{|E|}{4}-\frac{|V|}{2}$.
\end{proposition}

\begin{corollary}\label{corollary:genus}
    If $G$ is a finite group whose subgroup lattice is not a tree then 
    \begin{equation}\label{eq:genus_lower_bound_2}
    \gamma(\Gamma(G))\ge 1+\frac{|E|}{4}-\frac{|V|}{2}.
    \end{equation}
\end{corollary}

\begin{remark}
    The only abelian groups whose subgroup lattice is a tree are $\Z_{p^k}$.
\end{remark}

\begin{proof}
    We proceed by showing that a subgroup lattice graph has girth at least $4$. Clearly a subgroup lattice cannot have a cycle of length two since there is at most one edge between any two vertices. Could there be three subgroups $S_1,S_2,S_3$ that form the vertices of a $3$-cycle in the subgroup lattice? If so, one deduces that one subgroup, say $S_1$, is a subgroup of the other two. Furthermore, one of the remaining subgroups, say $S_2$, is a subgroup of the last subgroup, $S_3$. Hence $S_1<S_2<S_3$. This is a contradiction with the fact that $S_1$ and $S_3$ are connected by an edge in the subgroup lattice graph. Hence, we conclude that a $3$-cycle is impossible in a subgroup lattice graph.
\end{proof}

\begin{lemma}\label{lemma:Zp_Zp_Zp}
    If $p$ is a prime, then $\gamma(\Gamma(\Z_p\times \Z_p\times \Z_p))\ge 2$.
\end{lemma}
\begin{proof}
There are $p^2+p+1$ subgroups of size $p$ and $p^2+p+1$ subgroups of size $p^2$. Since a subgroup of size $p^2$ contains $p^2-1$ elements of order $p$, it thus contains $\frac{p^2-1}{p-1}=p+1$ subgroups of size $p$. Thus, each subgroup of size $p^2$ connects to $p+1$ subgroups of size $p$.

Using \eqref{eq:genus_lower_bound_2}, the subgroup lattice graph has $2(p^2+p+1)+2=2p^2+2p+4$ vertices and $2(p^2+p+1)+(p+1)(p^2+p+1)=p^3+4p^2+4p+3$ edges. Since the girth of the graph is $4$, 

$$
\gamma(\Gamma(\Z_p\times \Z_p\times \Z_p))\ge 1+\frac{p^3+4p^2+4p+3}{4}-\frac{2p^2+2p+4}{2}=\frac{1}{4}(p^3-1).
$$
This is strictly bigger than $1$ since $p\ge 2$.
\end{proof}

\begin{lemma}\label{lemma:eliminated_groups}
The subgroup lattices of the groups below have genus at least $2$.
\begin{itemize}
    \item $\Z_2\times \Z_2\times \Z_3\times \Z_3$
    \item $\Z_4\times \Z_4\times \Z_q$ for a prime $q\ne 2$
    \item $\Z_2\times \Z_2\times \Z_{q_1}\times \Z_{q_2}$ for distinct primes $q_1,q_2\ne 2$
    \item $\Z_3\times \Z_3\times \Z_{q_1}\times \Z_{q_2}$ for distinct primes $q_1,q_2\ne 3$
   
\end{itemize}

\end{lemma}

\begin{proof}
The graph $\Gamma(\Z_2\times \Z_2\times \Z_3\times \Z_3)$ has 30 vertices and 76 edges. Using \eqref{eq:genus_lower_bound_2}, one computes that $\gamma(\Gamma(\Z_2\times \Z_2\times \Z_3\times \Z_3))\ge 1+\frac{76}{4}-\frac{30}{2}=5.$

For a prime $q\ne 2$, the graph $\Gamma(\Z_4\times \Z_4\times \Z_q)$ has 30 vertices and 63 edges. Using \eqref{eq:genus_lower_bound_2}, one computes that $\gamma(\Gamma(\Z_4\times \Z_4\times \Z_q))\ge 1+\frac{63}{4}-\frac{30}{2}=1.75.$

For distinct primes $q_1,q_2\ne 2$, the graph $\Gamma(\Z_2\times \Z_2\times \Z_{q_1}\times \Z_{q_2})$ has 20 vertices and 44 edges. Using \eqref{eq:genus_lower_bound_2}, one computes that $\gamma(\Gamma(\Z_2\times \Z_2\times \Z_{q_1}\times \Z_{q_2}))\ge 1+\frac{44}{4}-\frac{20}{2}=2.$

For distinct primes $q_1,q_2\ne 3$, the graph $\Gamma(\Z_3\times \Z_3\times \Z_{q_1}\times \Z_{q_2})$ has 24 vertices and 56 edges. Using \eqref{eq:genus_lower_bound_2}, one computes that $\gamma(\Gamma(\Z_3\times \Z_3\times \Z_{q_1}\times \Z_{q_2}))\ge 1+\frac{56}{4}-\frac{24}{2}=3.$
\end{proof}

\begin{proposition}\label{prop:eliminated_groups}
    The following groups have subgroup lattice graphs with genus at least $2$.
    \begin{itemize}
        \item For a prime $q\ne 2$, $\Z_8\times \Z_2\times \Z_q$
        \item For a prime $q\ne 3$, $\Z_9\times \Z_3\times \Z_q$
        \item $\Z_{16}\times \Z_4$
        \item $\Z_8\times \Z_8$
         \item For a prime $q\ne 2$, $\Z_2\times\Z_2\times\Z_{q^2}$ and $\Z_4\times\Z_2\times\Z_{q^2}$
    \end{itemize}
\end{proposition}
\begin{proof}
To prove this proposition in every case except $\Z_4\times\Z_2\times\Z_{q^2}$, we will display minors with genus $2$ of the subgroup lattice graphs of the groups listed above. The graphs with the identified minors are displayed below in Figures \ref{fig:Z8_Z2_Zq}, \ref{fig:Z9_Z3_Zq}, \ref{fig:Z16_Z4}, \ref{fig:Z8_Z8}, and \ref{fig:Z2_Z2_Zp2}.

\begin{figure}[!h]
    \includegraphics[width=0.35\linewidth]{"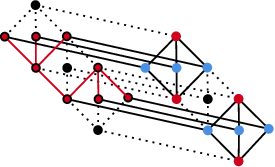"}
    \caption{Subgroup Lattice Graph of $\Z_8\times \Z_2\times \Z_q$ for a Prime $q\ne 2$}
    \label{fig:Z8_Z2_Zq}
\end{figure}

\begin{figure}[!h]
    \includegraphics[width=0.5\linewidth]{"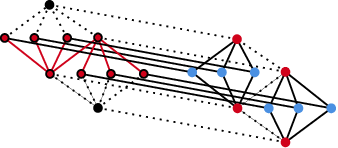"}
    \caption{Subgroup Lattice Graph of $\Z_9\times \Z_3\times \Z_q$ for a Prime $q\ne 3$}
    \label{fig:Z9_Z3_Zq}
\end{figure}

\begin{figure}[!h]
    \includegraphics[width=0.35\linewidth]{"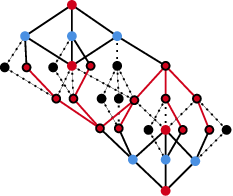"}
    \caption{Subgroup Lattice Graph of $\Z_{16}\times \Z_{4}$}
    \label{fig:Z16_Z4}
\end{figure}

\begin{figure}[!h]
    \includegraphics[width=0.35\linewidth]{"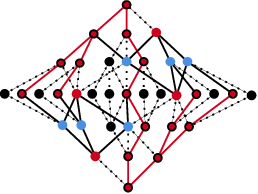"}
    \caption{Subgroup Lattice Graph of $\Z_8\times \Z_8$}
    \label{fig:Z8_Z8}
\end{figure}

\begin{figure}[!h]
    \includegraphics[width=0.35\linewidth]{"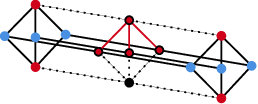"}
    \caption{Subgroup Lattice Graph of $\Z_2\times \Z_2\times \Z_{q^2}$ for a Prime $q\ne 2$}
    \label{fig:Z2_Z2_Zp2}
\end{figure}

In all of the subgroup lattice graph diagrams in Figures \ref{fig:Z8_Z2_Zq}, \ref{fig:Z9_Z3_Zq}, \ref{fig:Z16_Z4}, \ref{fig:Z8_Z8}, and \ref{fig:Z2_Z2_Zp2}, merging all the red vertices with a black outline and deleting all the dotted edges results in the following graph minor.

\begin{center}
\includegraphics[width=0.25\linewidth]{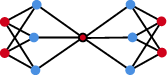}
\end{center}

By Theorem \ref{theorem:block} and the fact that $\gamma(K_{3,3})=1$, we have shown that each of the above graphs have genus at least $2$ except in the case of $\Z_4\times\Z_2\times\Z_{q^2}$. This result follows immediately from the fact that $\Gamma(\Z_2\times\Z_2\times\Z_{q^2})$ is a subgraph of $\Gamma(\Z_4\times\Z_2\times\Z_{q^2})$.
\end{proof}

\begin{lemma}\label{lemma:Zp_Zp_Zq2}
The genus of $\Gamma(\Z_p\times\Z_p\times \Z_{q^2})$ for primes $p\ne q$ is greater than or equal to $p-1$.
\end{lemma}

\begin{proof}
The subgroup lattice of $\Z_p\times \Z_p\times \Z_{q^2}$ is pictured below.

\begin{figure}[!h]
    \includegraphics[width=0.45\linewidth]{"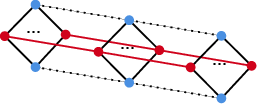"}
    \caption{Subgroup Lattice Graph of $\Z_p\times \Z_p\times \Z_{q^2}$}
\end{figure}
The vertices between each pair of red vertices represent $p+1$ subgroups, and there are $3p+3$ red vertices total. If we merge the red vertices in groups of three as shown in the picture and remove the dotted edges, we obtain a $K_{6,p+1}$ graph which has genus $\lceil \frac{(6-2)(p+1-2)}{4}\rceil=p-1$.
    
\end{proof}

\subsection{Proof of Theorem \ref{theorem:main}}
Now, we can proceed with the proof of the main theorem.

\begin{proof}[Proof of Theorem \ref{theorem:main}]
    Proposition \ref{prop:cyclic_genus_one} classifies all cyclic groups with subgroup lattice graphs of genus $1$ which are $\Z_{p_1^2p_2^2p_3}$, $\Z_{p_1^3p_2^2p_3}$, $\Z_{p_1^3p_2^3p_3}$ and $\Z_{p_1p_2p_3p_4}$ where the $p_i$ are distinct primes. 

    We now turn to the non-cyclic abelian groups. Proposition \ref{prop:non_cyclic_genus_one} demonstrates that every non-cyclic abelian group listed in Theorem \ref{theorem:main} has genus one. 
    
    Now, we must demonstrate that every non-cyclic abelian group not appearing in Theorem \ref{theorem:main} or Theorem \ref{theorem:genus_zero} has a subgroup lattice with genus at least two. 

    Let 
    \begin{equation}\label{eq:sum}
    G=\bigoplus_{i=1}^n{\Z_{p_i^{k_i}}}
    \end{equation}
    be a non-cyclic abelian group with a subgroup lattice group with genus $1$ \underline{that does not} appear in the list of Theorem \ref{theorem:main} or Theorem \ref{theorem:genus_zero}. Since $G$ is not cyclic, $p_i=p_j$ for some $i\ne j$ in \eqref{eq:sum}.

    First, let's consider the case where all the $p_i$ are the same in \eqref{eq:sum}. Lemma \ref{lemma:Zp_Zp_Zp} shows that $n=2$ necessarily. Because of Theorem \ref{theorem:genus_zero}, we only need to investigate the case $k_1,k_2\ge 2$. Proposition \ref{prop:Zp2_Zp2} shows us that $\gamma(\Gamma(\Z_{p^2}\times \Z_{p^2}))\ge 2$ when $p\ge 7$. Now, we need only consider the case where $p=2$, $3$, or $5$ and $k_1\ge 3$ and $k_2\ge 2$. Proposition \ref{prop:Zp3_Zp2} shows us that $\gamma(\Gamma(\Z_{p^3}\times \Z_{p^2}))\ge 2$ when $p\ge 3.$ So, we are reduced to considering cases where $p=2$ and $k_1\ge 4$ and $k_2\ge 2$ \underline{or} $k_1\ge 3$ and $k_2\ge 3$. Proposition \ref{prop:eliminated_groups} shows that these groups have subgroup lattices with genus at least $2$. Hence, Theorem \ref{theorem:main} contains the complete classification of non-cyclic abelian group with subgroup lattice with genus equal to one in the case where all the $p_i$ are the same in \eqref{eq:sum}.

    Now, we consider the case where the collection of primes $\{p_i\}$ in \eqref{eq:sum} contains exactly two primes. Recall that at least one of the primes must be repeated in order for $G$ to be non-cyclic. Proposition \ref{prop:Zp_Zp_Zq} shows that if $\gamma(\Gamma(G))=1$, then the prime repeated must be $2$ or $3$. 
    
    Consider the case where $p_1=p_2=3$. Proposition \ref{prop:eliminated_groups} forces $k_1=k_2=1$. Since $G$ cannot appear in the list of Theorem \ref{theorem:main}, if $n=3$, then $k_3>1$ necessarily. However, Lemma \ref{lemma:Zp_Zp_Zq2} rules out this case. Hence, we restrict ourselves to the case $n=4$ which implies $p_3=p_4=2$ necessarily. However, Lemma \ref{lemma:eliminated_groups} rules out this possibility.

    Now consider the case where $p_1=p_2=2$. Proposition \ref{prop:eliminated_groups} forces either $k_1=k_2=1$ or $k_1=2,k_2=1$. Since $G$ cannot appear in the list of Theorem \ref{theorem:main}, if $n=3$, then $k_3>1$ necessarily. However, Proposition \ref{prop:eliminated_groups} rules out this case. Hence, we restrict ourselves to the case $n=4$ which implies $p_3=p_4=3$ necessarily. However, Lemma \ref{lemma:eliminated_groups} rules out this possibility.
    
    Hence, Theorem \ref{theorem:main} contains the complete classification of non-cyclic abelian group with subgroup lattice with genus equal to one in the case that collection of primes $\{ p_i\}$ in \eqref{eq:sum} contains exactly two distinct primes.

    Lastly, we consider the case where the collection of primes $\{p_i\}$ in \eqref{eq:sum} contains at least three distinct primes. Proposition \ref{prop:Zp_Zp_Zq} shows that if $\gamma(\Gamma(G))=1$, then the prime repeated must be $2$ or $3$. However, Lemma \ref{lemma:eliminated_groups} shows that $\gamma(\Gamma(G))\ge 2$ in this case since our group $G$ would contain either $\Z_2\times \Z_2\times \Z_{q_1}\times \Z_{q_2}$ or $\Z_3\times \Z_3\times \Z_{q_1}\times \Z_{q_2}$ as a subgroup.

    Hence, Theorem \ref{theorem:main} contains the complete classification of non-cyclic abelian group with subgroup lattice with genus equal to one.
\end{proof}

\section{Conclusion}
The prospect of writing out explicit lists of abelian groups whose subgroup lattice graphs have genus $g$ for every $g$ seems daunting. However, there are various classes of abelian groups for which this task seems possible. In particular, the authors of \cite{MS22} have a conjecture for the genus of the subgroup lattice graphs of cyclic groups with three distinct primes factors. In particular, they conjecture that the inequalities in Proposition \ref{prop:upper_bounds} are actually equalities. Further investigation is warranted.

\section{Acknowledgements}
All the subgroup lattice graphs were constructed using SAGE \cite{sagemath}. The author would like to thank Colin Starr for introducing him to this problem in 2017 while the author was a Visiting Assistant Professor at Willamette University.

\bibliography{main}
\bibliographystyle{plain}
\end{document}